\documentclass[reqno]{amsart}
\usepackage{amssymb}
\usepackage[dvips]{epsfig}
\usepackage{graphicx}
\usepackage{color}
\usepackage{amsmath}
\usepackage{amssymb}
\usepackage{amsfonts}
\usepackage{amsthm}
\usepackage{bbm}
\usepackage{tikz}
\usepackage[latin1]{inputenc}

\newcommand{\Z}{\mathbb Z}

\newcommand{\N}{\mathbb{N}}

\newtheorem{thm}{Theorem}[section]
\newtheorem{lem}[thm]{Lemma}
\newtheorem{prop}[thm]{Proposition}

\theoremstyle{remark}
\newtheorem{rem}{\bf Remark}[section]
\theoremstyle{definition}
\newtheorem{defn}[thm]{Definition}
\newtheorem{conj}[thm]{Conjecture}

\numberwithin{equation}{section}

\begin{document}
\title[Anderson localization for NLSE]{Long-time Anderson Localization for the  Nonlinear Schr\"odinger Equation Revisited}

\author[H.C.]{Hongzi Cong}
\address[H. Cong]{School of Mathematical Sciences, Dalian University of Technology, Dalian 116024, China}
\email{conghongzi@dlut.edu.cn}

\author[Y.S.]{Yunfeng Shi}
\address[Y. Shi] {College of Mathematics,
Sichuan University,
Chengdu 610064,
China}
\email{yunfengshi@scu.edu.cn, yunfengshi18@gmail.com}
\author[Z.Z.]{Zhifei Zhang}
\address[Z. Zhang] {School of Mathematical Sciences,
Peking University,
Beijing 100871,
China}
\email{zfzhang@math.pku.edu.cn}

\date{\today}



\begin{abstract}
In this paper, we confirm the conjecture of Wang and Zhang ({J. Stat. Phys. 134 (5-6):  953--968, 2009}) in a long time scale, i.e.,
 the displacement of the wavefront for $1D$ nonlinear random Schr\"odinger equation is of logarithmic order in time $|t|$.
\end{abstract}

\maketitle

\section{Introduction}
Anderson localization  was originally discussed by Anderson \cite{And58} in the context of  wave
propagation of  non-interacting quantum particles through random disordered media. Since this seminal work,  a great deal of attention has been paid to this topic both in physics and mathematics community. The Anderson model is a discrete linear Schr\"odinger operator defined on $\ell^2(\mathbb{Z}^d)$
\begin{align}
H_0=-\epsilon_1\Delta+\lambda v_n(\omega)\delta_{nn'},
\end{align}
where $\Delta$ is the discrete Laplacian: $(\Delta q)_n=\sum\limits_{|e|_1=1}q_{n+e}$ (with $|e|_1=\sum\limits_{i=1}^d|e_i|$) and $(v_{n}(\omega))_{n\in \mathbb{Z}^d}$ is  a family of identical independent distributed (i.i.d.) random variables with uniform distribution on $[0,1]$ (i.e., ${ d}v_n(\omega)=\chi_{[0,1]}(\omega_n){d}\omega_n$). The constant $\epsilon_1\geq 0$  is the coupling for describing the strength of random disorder. We say that $H_0$ has Anderson localization (AL) if its spectrum is pure point with exponentially decaying eigenfunctions. In many cases, we are  interested in the dynamics of the time dependent (linear) Schr\"odinger equation  associated with  $H_0$
\begin{align}\label{lse}
\mathbf{i}\dot{q}=H_0q,
\end{align}
where $q\in \ell^2(\mathbb{Z}^d)$. The standard spectral theorem of self-adjoint operators deduces that \eqref{lse} has a unique global solution $q(t)=e^{-\mathbf{i}t H_0}q(0)$ for each initial data $q(0)\in\ell^2(\mathbb{Z}^d)$. The evolution operator $e^{-\mathbf{i}tH_0}$ is unitary for each $t\in \mathbb{R}$ and thus preserves the $\ell^2$-norm. If we want to know more precise information about the wave packet propagation, we can introduce the concept of dynamical localization (DL) for $H_0$: for any $\alpha>0$ and $q(0)\in\ell^2(\mathbb{Z}^d)$,
\begin{align}\label{dl}
\sup_{t\in \mathbb{R}}\sum_{n\in \mathbb{Z}^d}(1+|n|)^{2\alpha}|q_n(t)|^2<\infty,
\end{align}
where $q(t)=e^{-\mathbf{i}tH_0}q(0)$ and $|n|=\max\limits_{1\leq i\leq d}|n_i|$. The DL implies that the particle is concentrated near the origin uniformly for all time.

The first mathematical rigorous proof of localization for random operators was due to Goldsheid-Molchanov-Pastur \cite{GMP77} for $1D$ continuous random Schr\"odinger operators. In high dimensions,  Fr\"ohlich-Spencer \cite{FS83} proved, either at high disorder (i.e., $\epsilon_1\ll1$) or low energy, the absence of diffusion for Anderson model by developing the celebrated multi-scale analysis (MSA) method. Based on MSA of \cite{FS83},  \cite{FMSS85,DS85,SW86} finally obtained the Anderson localization at either high disorder or low energy. An alternative method for the proof of localization for random operators, known as the fractional moment method (FMM), was developed by Aizenman-Molchanov \cite{AM93}.  Remarkably, by employing FMM, Aizenman proved the first DL for Anderson model  \cite{Aiz94}.

When a nonlinear perturbation is added in \eqref{lse}, we are led to the study of the so called nonlinear Schr\"odinger equations with a random potential. In this paper, we focus on the following $1D$ nonlinear Schr\"odinger equation (NLSE)
\begin{equation}\label{010701}
\mathbf{i} \dot{q}_j=\epsilon_1(q_{j-1}+q_{j+1})+v_j(\omega)q_j+\epsilon_2|q_j|^2q_j,
\end{equation}
and in particular the solution  $q(t)$ of \eqref{010701} with an initial  state $q(0)\in\ell^2(\mathbb{Z})$ as $t\to\infty$. The NLSE also has  important applications in a variety of physical systems, especially the Bose-Einstein Condensation \cite{DGPS} (we refer  to \cite{FKS12} for  an excellent review on NLSE). Since in nonlinear case the spectral theorem becomes invalid,  the  study of AL for a NLSE seems vacuous. However, in linear case the famous RAGE Theorem (see \cite{K08} and \cite{AW15}) claims that $H_0$ has pure point spectrum if and only if, for any $q(0)\in\ell^2(\mathbb{Z})$,
\begin{align}\label{rage}
\lim_{N\to\infty}\sup_{t\in \mathbb{R}}\sum_{|j|>N}|q_j(t)|^2=0,
\end{align}
where $q(t)=e^{-\mathbf{i}tH_0}q(0)$. Thus, it is natural to define AL for a NLSE via \eqref{rage} by noting that \eqref{010701} is globally well-posed for any initial data belongs to $\ell^2(\mathbb{Z})$.

{The numerical results found by Pikovsky-Shepelyansky \cite{10}  and by Flach and coworkers \cite{12,14,67,66} suggested that an initially localized wavepacket  spreads eventually in the presence of nonlinearity. Particularly, in the weak nonlinearity case (i.e., $0<\epsilon_2\ll1$), it was numerically established in \cite{12} that AL occurs up to some time scale $T_{\epsilon_2}>0$  which increases with decreasing $\epsilon_2$. Moreover, for $t>T_{\epsilon_2}$,  the wavepacket starts to spread sub-diffusively. However, all rigorous theories predict that the spreading cannot be faster than logarithmic in time. This seems due to the fact that numerical calculations for chaotic systems are quite sensitive to numerical errors (see \cite{FKS12} for details).}

The first rigorous result towards nonlinear AL for NLSE with i.i.d. random potential was obtained by Fr\"ohlich-Spencer-Wayne \cite{FSW86}: they showed that, with high probability and weak nonlinearity, any  \textit{sup-exponentially}  localized  initial state always	stayed in a full dimensional KAM  tori.  Their proof is based on an extension of the KAM techniques. Later, if the initial state is \textit{polynomially} localized, by using Birkhoff normal form method, Benettin-Fr\"ohlich-Giorgilli \cite{BFG88} got that  the propagation remains localized  in very \textit{long-time}  for some $dD$  lattice nonlinear oscillation equations with i.i.d. Gaussian random potential.  Recently, Bourgain-Wang \cite{BW08} constructed many quasi-periodic solutions for some random NLSE by combining Nash-Moser iteration and the improved MSA. {We would also like to mention the works of Yuan \cite{Yua02} and Geng-You-Zhao \cite{GYZ14}\footnote{\cite{GYZ14} studied the  NLSE with quasi-periodic potentials.}, in which the persistence of quasi-periodic solutions for some $1D$ discrete nonlinear equations was proved  via the KAM type iterations scheme.}

The most important result for nonlinear AL with  non-localized initial state was due to  Wang-Zhang \cite{WZ09}: they proved the first  ``truly'' \textit{long-time} AL for the $1D$ NLSE. More precisely, they established that {\it Given $A\geq 2,\delta>0$, let $q(0)\in\ell^2(\mathbb{Z})$ be any initial state satisfying $\sum\limits_{|j|>j_0>0}|q_j(0)|^2\leq \delta$. Then there exist $\varepsilon=\varepsilon(A)>0, C=C(A)>0$ such that for $0<\epsilon=\epsilon_1+\epsilon_2\leq \varepsilon$ and $t\leq \delta C^{-1}\epsilon^{-A}$,
\begin{align*}
\sum\limits_{|j|> j_0+N}|q_j(t)|^2\leq 2\delta
\end{align*}
with probability at least $1-\exp\left(-\frac{j_0}{N}e^{-2N\epsilon^{(CA)^{-1}}}\right)$ and {$N=N(A)\geq A^2$}.
}
In this theorem, they required actually both high disorder  and weak nonlinearity.  The proof depends on some type of Birkhoff normal form borrowed from Bourgain-Wang \cite{BW07}.  Remarkably, Fishman-Krivolapov-Soffer \cite{FKS08} obtained the \textit{long-time} \textit{exponentially} DL (i.e., with $(1+|n|)^{2\alpha}$ being replaced by the exponential bound in \eqref{dl}) of time $t\leq \epsilon_2^{-2}$ under just weak nonlinearity assumption. Their proof differs from that of Wang-Zhang and is based on perturbation theory combined with FMM of Aizenman-Molchanov \cite{AM93}. Subsequently, some results of \cite{FKS08} have been improved to time of order $\epsilon_2^{-A}$ for any $A\geq2$ \cite{FKS09} by the same authors, but the proof is  partly rigorous:  in some parts it relies on conjectures that they tested numerically.

Wang-Zhang's result mentioned as above indicates that if $\epsilon\leq \varepsilon(A)$,  AL holds of time scale $T_\epsilon\sim\epsilon^{-A}$, and as a result the wavefront $N$ depends on time in the following way
\begin{align}
N\sim (\ln T_\epsilon)^{2+}.
\end{align}
In addition, it was proven in \cite{W08} that the growth of Sobolev norms is at most logarithmic in $t$. These enable them to raise  the conjecture:
\begin{conj}[\cite{WZ09}]
As $t\rightarrow \infty, $ the displacement of the wavefront $N$ is of order $t^{0^{+}}$ (possibly logarithmic).
\end{conj}
The main motivation of the present paper comes from this conjecture. In fact, we prove the following main result.
\begin{thm}\label{main}
Given $\delta>0$, for all initial datum $q(0)\in\ell^2(\mathbb{Z})$, let $j_0\in\mathbb{N}$ be such that
\begin{equation*}
\sum_{|j|>j_0}|q_j(0)|^2<\delta.
\end{equation*}
Fix  $0<\alpha<1/100.$ Then there exists constant $\varepsilon=\varepsilon(\alpha)>0$ such that the following holds: for $0<\epsilon:=\epsilon_1+\epsilon_2<\varepsilon$ and
for all
$$|t|\leq \delta \exp\left(\frac{\left|\ln \epsilon\right|^2}{200\ln\left|\ln\epsilon\right|}\right)$$ one has
\begin{equation*}
\sum_{|j|>j_0+N}|q_{j}(t)|^2<2\delta
\end{equation*}
with probability at least
\begin{align*}
1-{\epsilon^{\alpha/2}},
\end{align*}
where
\begin{equation*}
N=\left|\frac{\ln\epsilon}{200\ln\left|\ln\epsilon\right|}\right|^2.
\end{equation*}
\end{thm}
\begin{rem}
\begin{itemize}
\item[]
\item[$\bullet$] As an easy corollary, one has for $|t|\leq T_\epsilon=\delta \exp\left(\frac{\left|\ln \epsilon\right|^2}{200\ln\left|\ln\epsilon\right|}\right)$,
\begin{align*}
N(\epsilon)\sim\ln T_\epsilon.
\end{align*}
Moreover, $T_\epsilon\to\infty$ in the exponential rate as $\epsilon\to 0$. This  confirms Wang-Zhang's conjecture in a long time scale.

\item[$\bullet$] Our result can't be derived directly from Wang-Zhang's by choosing $A\sim\frac{\left|\ln \epsilon\right|}{\ln\left|\ln\epsilon\right|}.$  It is because the perturbation $\varepsilon(A)$ in their argument depends sensitively on $A$. In order to improve Wang-Zhang's \textit{polynomial} bound to the \textit{exponential} one, it requires new ideas.
\end{itemize}
\end{rem}

We then outline the proof. The main scheme of our proof  is definitely adapted from Bourgain-Wang \cite{BW07} and Wang-Zhang \cite{WZ09},  which  uses Birkhoff normal form type transformations to
construct barriers centered at some $\pm j_0, j_0>1$ of width $N$, where the terms responsible for propagation are small enough. However, while our localized time is significantly much longer, our
argument can also be viewed as both a clarification and at the same time streamlining of \cite{WZ09}.  This is due to several important technical improvements that we add to Wang-Zhang's scheme:
\begin{itemize}
\item[(1)] One important highlight is that, we make use of $\ell^1$-norm (with an exponential weight) rather than $\ell^\infty$-norm for the Hamiltonian. This will lead to more clear and effective estimate on some key ingredients, such as the Poisson bracket, symplectic transformations and particularly the small divisors when performing the Birkhoff normal form. In addition, we  deal with those elements in  a separated fashion, which makes the proof  more tractable.
\item[(2)] Another issue we want to highlight is that we introduce new ideas originated from  Benettin-Fr\"ohlich-Giorgilli \cite{BFG88} in our proof. In the iteration scheme, we always assume that both the width $N$ of the barriers and the total iteration steps $M$  are  \textit{non-negligible} as compared with the perturbation $\epsilon$.   Then our main result follows from  \textit{optimal} choices of $N,M$ depending on $\epsilon$. To achieve this goal, one needs to take care of all terms in the barriers and thus needs to use the $\ell^1$-norm.
\end{itemize}

The structure of the paper is as follows.  Some important facts on Hamiltonian dynamics, such as the Poisson bracket, symplectic transformation and non-resonant conditions  are presented in \S2. The Birkhoff normal form type theorem is proved in \S3.  The estimate on the probability  when handling the small divisors can be found in \S4.  The proof of our main theorem is finished in \S5.

\section{Structure of the transformed Hamiltonian}
We recast (\ref{010701}) as a Hamiltonian equation
\begin{align*}
\mathbf{i} \dot{q}_j=2\frac{\partial H}{\partial \bar{q}_j},
\end{align*}
where
\begin{align}\label{010703}
H(q,\bar q)=\frac{1}2\left(\sum_{j\in\mathbb{Z}}v_j|q_j|^2+\epsilon_1\sum_{j\in\mathbb{Z}}\left(\bar q_jq_{j+1}+q_j\bar q_{j+1}\right)+\frac12\epsilon_2\sum_{j\in\mathbb{Z}}|q_j|^4\right).
\end{align}
As is well-known, the $\ell^2$-norm of the solution $q(t)$ is conserved, i.e.,
\begin{align*}
\sum_{j\in\mathbb{Z}}\left|q_j(t)\right|^2=\sum_{j\in\mathbb{Z}}\left|q_j(0)\right|^2\quad {\rm for}\  \forall\  t\in\mathbb{R}.
\end{align*}
In order to prove the main result, we need to control the time derivative of the truncated sum of higher modes
\begin{align}\label{022101}
\frac{d}{dt}\sum_{|j|>j_0}\left|q_{j}(t)\right|^2.
\end{align}
In what follows, we will deal extensively with monomials in $q_j$.  Rewrite any monomials in the form
\begin{align}\label{mono}
\prod_{j\in\mathbb{Z}}q_j^{n_j}\bar q_{j}^{n_j'}.
\end{align}
Let
\begin{align*}
n=(n_j,n_j')_{j\in\mathbb{Z}}\in \mathbb{N}^{\mathbb{Z}}\times\mathbb{N}^{\mathbb{Z}}.
\end{align*}
We  define
\begin{align*}
{\rm supp}\  n&=\{j\in\mathbb{Z}: \ n_j\neq 0\ \mbox{or}\ n_j'\neq 0\},\\
\Delta(n)&=\sup_{j,j'\in{\rm supp}\  n }|j-j'|,\\
|n|&=\sum_{j\in\mathbb{Z}}(n_j+n_j').
\end{align*}
If $n_j=n_j'$ for all $j\in\mbox{supp}\ n$, then the monomial \eqref{mono} is called resonant. Otherwise it is called non-resonant. Note that non-resonant monomials contribute to the truncated sum in (\ref{022101}), where resonant ones do not. We define the (resonant) set as
\begin{align}\label{030301}
\mathcal{N}=\left\{n\in\mathbb{N}^{\mathbb{Z}}\times\mathbb{N}^{\mathbb{Z}}:\ n_j=n_j'\ {\rm for}\ \forall\  j\right\}.
\end{align}
Given $j_0$ and $N\in\mathbb{N}$, let
\begin{align*}
A(j_0,N):=\left[j_0-N,j_0+N\right]\cup\left[-j_0-N,-j_0+N\right].
\end{align*}

\begin{defn}\label{122403}
Given a Hamiltonian
\begin{align*}
H(q,\bar{q})=\sum_{n\in\mathbb{N}^{\mathbb{Z}}\times\mathbb{N}^{\mathbb{Z}}}H(n)\prod_{{\rm supp}\ n}q_j^{n_j}\bar {q}_j^{n_j'},
\end{align*}
for $j_0,N\in\mathbb{N}$ and $r>2$, we define
\begin{align}\label{122402}
\left\| H\right\|_{j_0,N,r}=\sum_{n\in\mathbb{N}^{\mathbb{Z}}\times\mathbb{N}^{\mathbb{Z}}\atop
{{\rm supp}}\ n\cap A(j_0,N)\neq \emptyset}{\left|H(n)\right|\cdot|n|\cdot r^{\Delta(n)+|n|-1}}
\end{align}
and
\begin{align}\label{032302}
\left\| H\right\|_{j_0,N,r}^{\mathcal{L}}=\sup_{j\in\mathbb{Z}}\sum_{n\in\mathbb{N}^{\mathbb{Z}}\times\mathbb{N}^{\mathbb{Z}}\atop
{\rm supp}\ n\cap A(j_0,N)\neq \emptyset}{\left|\partial_{v_j}H(n)\right|\cdot|n|\cdot r^{\Delta(n)+|n|-1}},
\end{align}
where $v=(v_j)_{j\in\mathbb{Z}}$ is the potential.
Define
\begin{align*}
\left|\left| \left|H\right|\right|\right|_{j_0,N,r}=\left\| H\right\|_{j_0,N,r}+\left\| H\right\|_{j_0,N,r}^{\mathcal{L}}.
\end{align*}
\end{defn}

\begin{defn}
Given
\begin{align*}
H(q,\bar{q})=\sum_{n\in\mathbb{N}^{\mathbb{Z}}\times\mathbb{N}^{\mathbb{Z}}}H(n)\prod_{{\rm supp}\ n}q_j^{n_j}\bar {q}_j^{n_j'}
\end{align*}
and
\begin{equation*}
G(q,\bar{q})=\sum_{m\in\mathbb{N}^{\mathbb{Z}}\times\mathbb{N}^{\mathbb{Z}}}G(m)\prod_{{\rm supp}\ m}q_j^{m_j}\bar {q}_j^{m_j'},
\end{equation*}
the Poisson bracket of $H$ and $G$ is defined as
\begin{equation*}\label{022201}
\{H,G\}:=\mathbf{i}\sum_{n,m\in\mathbb{N}^{\mathbb{Z}}\times\mathbb{N}^{\mathbb{Z}}}\sum_{k\in\mathbb{Z}}H(n)G(m)(n_km_k'-n_k'm_k)q_k^{n_k+m_k-1}\bar {q}_k^{n_k'+m_k'-1}\left(\prod_{j\neq k}q_j^{n_j+m_j}\bar {q}_j^{n_j'+m_j'}\right).
\end{equation*}
\end{defn}

We have the following key estimate.

\begin{prop}[\textbf{Poisson Bracket}]\label{090603}
For $j_0,N\in\mathbb{N}$, let $a$ and $b$ satisfy $$[a,b]\subset [j_0-N,j_0+N].$$
Let
\begin{align*}
H(q,\bar{q})=\sum_{n\in\mathbb{N}^{\mathbb{Z}}\times\mathbb{N}^{\mathbb{Z}}}H(n)\prod_{{\rm supp}\ n}q_j^{n_j}\bar {q}_j^{n_j'}
\end{align*}
and
 \begin{align*}
G(q,\bar{q})=\sum_{m\in\mathbb{N}^{\mathbb{Z}}\times\mathbb{N}^{\mathbb{Z}}}G(m)\prod_{{\rm supp}\ m}q_j^{m_j}\bar {q}_j^{m_j'}
\end{align*}
with
\begin{align}\label{122401}
{\rm supp}\ n\subset[-b,-a]\cup[a,b]\quad \text{for any}\quad  n.
\end{align}
Then for any $0<\sigma<r/2$, we have
\begin{align}\label{122002}
\left|\left|\left|\left\{ H, G\right\}\right|\right|\right|_{j_0,N,r-\sigma}\leq \frac{1}\sigma\left|\left|\left|H\right|\right|\right|_{j_0,N,r}\cdot|||G|||_{j_0,N,r}.
\end{align}
\end{prop}

\begin{proof}
First of all, we write
\begin{align*}
\{H,G\}=\sum_{l\in\mathbb{N}^{\mathbb{Z}}\times\mathbb{N}^{\mathbb{Z}}}\{H,G\}
(l)\prod_{{\rm supp}\ l}q_j^{l_j}\bar{q}_{j}^{l_j'},
\end{align*}
where
\begin{align}\label{011401}
\{H,G\}
(l)=\mathbf{i}\sum_{k\in\mathbb{Z}}\left(\sum_{n,m\in\mathbb{N}^{\mathbb{Z}}\times\mathbb{N}^{\mathbb{Z}}}^*H(n)G(m)\left(n_km'_k-n_k'm_k\right) \right)
\end{align}
and the sum $\sum\limits_{n,m\in\mathbb{N}^{\mathbb{Z}}\times\mathbb{N}^{\mathbb{Z}}}^*$  is taken as
\begin{align*}
&l_j=n_j+m_j-1, \quad l'_{j}=n'_j+m'_j-1\ {\rm for}\ j=k,\\
&l_j=n_j+m_j,\quad  l'_{j}=n'_j+m'_j\  {\rm for}\ j\neq k.
\end{align*}

Secondly, let
\begin{align*}
\widetilde G=\sum_{m\in\mathbb{N}^{\mathbb{Z}}\times\mathbb{N}^{\mathbb{Z}}\atop
{\rm supp}\ m\cap A(j_0,N)=\emptyset}G(m)\prod_{{\rm supp}\ m}q_j^{m_j}\bar {q}_j^{m'_j}
\end{align*}
and then following \eqref{122401}, one has $\left\{H,\widetilde G\right\}=0.$
Hence, we always assume that
\begin{align}\label{022401}
G=\sum_{m\in\mathbb{N}^{\mathbb{Z}}\times\mathbb{N}^{\mathbb{Z}}\atop
{\rm supp}\ m\cap A(j_0,N)\neq \emptyset}G(m)\prod_{{\rm supp}\ m}q_j^{m_j}\bar {q}_j^{m_j'}.
\end{align}

Without loss of generality, we assume that $H$ and $G$ are homogeneous polynomials with degrees $n^*$ and $m^*$ respectively, i.e.,
\begin{align*}
H(q,\bar{q})=\sum_{n\in\mathbb{N}^{\mathbb{Z}}\times\mathbb{N}^{\mathbb{Z}}\atop |n|=n^*}H(n)\prod_{{\rm supp}\ n}q_j^{n_j}\bar {q}_j^{n_j'}
\end{align*}
and
\begin{align*}
G=\sum_{m\in\mathbb{N}^{\mathbb{Z}}\times\mathbb{N}^{\mathbb{Z}},|m|=m^*\atop
{\rm supp}\ m\cap A(j_0,N)\neq \emptyset}G(m)\prod_{{\rm supp}\ m}q_j^{m_j}\bar {q}_j^{m_j'}.
\end{align*}

Since $r>2$ and $0<\sigma<r/2$, one has
\begin{align}\label{022202}
1<r-\sigma<r.
\end{align}
In view of \eqref{022202} and
\begin{align*}
\Delta (l)\leq \Delta(n)+\Delta(m),
\end{align*} one has
\begin{align}
&\nonumber\sum_{l\in\mathbb{N}^{\mathbb{Z}}\times\mathbb{N}^{\mathbb{Z}}}
\left| \sum_{k\in\mathbb{Z}} \sum_{n,m\in\mathbb{N}^{\mathbb{Z}}\times\mathbb{N}^{\mathbb{Z}}}^*H(n)G(m)\left(n_km'_k-n_k'm_k\right) \right|\left(r-\sigma\right)^{\Delta(l)} \\
\leq&\nonumber\sum_{n,m\in\mathbb{N}^{\mathbb{Z}}\times\mathbb{N}^{\mathbb{Z}}}
 \left|H(n)\right|\left|G(m)\right|\sum_{k\in\mathbb{Z}}\left(n_km'_k+n_k'm_k\right) \left(r-\sigma\right)^{\Delta(n)+\Delta(m)} \\
\leq&\label{022203}\left(\sum_{n\in\mathbb{N}^{\mathbb{Z}}\times\mathbb{N}^{\mathbb{Z}}}|H(n)|\cdot |n|\cdot r^{\Delta(n)}\right)
\left(\sum_{m\in\mathbb{N}^{\mathbb{Z}}\times\mathbb{N}^{\mathbb{Z}}}|G(m)|\cdot |m|\cdot r^{\Delta(m)}\right).
\end{align}
In view of (\ref{011401}), (\ref{022203}) and using
$|l|=|n|+|m|-2,$ we have
\begin{align*}
&\left\|\left\{H,G\right\}\right\|_{j_0,N,r-\sigma}\\
\leq&\left(|n|+|m|-2\right)(r-\sigma)^{|n|+|m|-3}\\
\ \ &\times\left(\sum_{n\in\mathbb{N}^{\mathbb{Z}}\times\mathbb{N}^{\mathbb{Z}}}|H(n)|\cdot |n|\cdot r^{\Delta(n)}\right)
\left(\sum_{m\in\mathbb{N}^{\mathbb{Z}}\times\mathbb{N}^{\mathbb{Z}}}|G(m)|\cdot |m|\cdot r^{\Delta(m)}\right)\\
\leq&\frac1{\sigma}\left(\sum_{n\in\mathbb{N}^{\mathbb{Z}}\times\mathbb{N}^{\mathbb{Z}}}|H(n)|\cdot |n|\cdot r^{\Delta(n)+|n|-1}\right)
\left(\sum_{m\in\mathbb{N}^{\mathbb{Z}}\times\mathbb{N}^{\mathbb{Z}}}|G(m)|\cdot |m|\cdot r^{\Delta(m)+|m|-1}\right),
\end{align*}
where the last inequality is based on
\begin{align*}
(|n|+|m|-2)(r-\sigma)^{|n|+|m|-3}\leq \frac{1}{\sigma}r^{|n|+|m|-2}.
\end{align*}
Using (\ref{122401}), (\ref{022401}) and Definition \ref{122403}, we have
\begin{align}
\label{042001}\left\|\left\{ H, G\right\}\right\|_{j_0,N,r-\sigma}\leq \frac{1}\sigma\left\|H\right\|_{j_0,N,r}\cdot\left\|G\right\|_{j_0,N,r}.
\end{align}
Finally, recalling
\begin{align*}
\partial_{v_j}\left(H(n)G(m)\right)=\partial_{v_j}H(n)\cdot G(m)+H(n)\cdot \partial_{v_j}G(m)
\end{align*}
and following the proof of (\ref{042001}), one has
\begin{align}\label{042002}
\left\|\left\{ H, G\right\}\right\|_{j_0,N,r-\sigma}^{\mathcal{L}}\leq \frac{1}\sigma\left(\left\|H\right\|_{j_0,N,r}^{\mathcal{L}}\cdot\left\|G\right\|_{j_0,N,r}+\left\|H\right\|_{j_0,N,r}\cdot\left\|G\right\|_{j_0,N,r}^{\mathcal{L}}\right).
\end{align}
Combining (\ref{042001}) and (\ref{042002}), we finish the proof of (\ref{122002}).
\end{proof}

\begin{prop}\label{E1}
{Let $H$ and $G$ be as in } Proposition \ref{090603}.  Assume further that
\begin{align}\label{042801}
\left(\frac{e}{\sigma}\right)|||H|||_{j_0,N,r}\leq \frac12.
\end{align}
Then
\begin{align*}
\left|\left|\left|G\circ X_H^1\right|\right|\right|_{j_0,N,r-\sigma}
\leq2\left|\left|\left|G\right|\right|\right|_{j_0,N,r},
\end{align*}
where $X_H^1$ is the time-$1$ map generated by the flow of $H$.
\end{prop}

\begin{proof}
 First of all,  we expand $G\circ X_H^1$ into the Taylor series
 \begin{align}\label{3.2}
 G\circ X_H^1=\sum_{n\geq 0}\frac{1}{n!}G^{(n)},
 \end{align}
where $G^{(n)}=\left\{G^{(n-1)},H\right\}$ and $G^{(0)}=G$.
We will estimate $\left|\left|\left|G^{(n)}\right|\right|\right|_{j_0,N,r-\sigma}$ by repeatedly using of Proposition \ref{090603}:
\begin{align}
\nonumber\left|\left|\left|G^{(n)}\right|\right|\right|_{j_0,N,r-\sigma}
\nonumber&=\left|\left|\left|\left\{G^{(n-1)},H\right\}\right|\right|\right|_{j_0,N,r-\sigma}\\
\nonumber&\leq\left(\frac{n}{\sigma}\right)\left(\left|\left|\left|H\right|\right|\right|_{j_0,N,r}\right)\left|\left|\left|G^{(n-1)}\right|\right|\right|_{j_0,N,r-\frac{(n-1)\sigma}{n}}\\
\nonumber&\leq\left(\frac{n}{\sigma}\right)^2\left(\left|\left|\left|H\right|\right|\right|_{j_0,N,r}\right)^2\left|\left|\left|G^{(n-2)}\right|\right|\right|_{j_0,N,r-\frac{(n-2)\sigma}{n}}\\
&\dots\nonumber\\
\nonumber&\leq\left(\frac{n}{\sigma}\right)^n\left(\left|\left|\left|H\right|\right|\right|_{j_0,N,r}\right)^n\left|\left|\left|G\right|\right|\right|_{j_0,N,r}.
\end{align}
Then
\begin{align}\label{122509}
\frac{1}{n!}\left|\left|\left|G^{(n)}\right|\right|\right|_{j_0,N,r-\sigma}\leq \left(\frac{e\left|\left|\left|H\right|\right|\right|_{j_0,N,r}}{\sigma}\right)^n\left|\left|\left|G\right|\right|\right|_{j_0,N,r},
\end{align}
where we use the inequality $n^n<n!e^n$.
Hence combining (\ref{3.2}) and (\ref{122509}), we obtain
\begin{align*}
\left|\left|\left|G\circ X_H^1\right|\right|\right|_{j_0,N,r-\sigma}
&\leq\sum_{n\geq 0}\left(\frac{e\left|\left|\left|H\right|\right|\right|_{j_0,N,r}}{\sigma}\right)^n\left|\left|\left|G\right|\right|\right|_{j_0,N,r}\\
&\leq  2
\left|\left|\left|G\right|\right|\right|_{j_0,N,r},
\end{align*}
where the last inequality is based on (\ref{042801}).
\end{proof}

\begin{rem}\label{092403}
In general, we have
\begin{align}\label{9}
\left|\left|\left|G\circ X_{H}^1-G\right|\right|\right|_{j_0,N,r-\sigma}\leq\frac{e}{\sigma}\cdot
\left|\left|\left|H\right|\right|\right|_{j_0,N,r}\cdot
\left|\left|\left|G\right|\right|\right|_{j_0,N,r},
\end{align}
and
\begin{align}\label{10}
\left|\left|\left|G\circ X_H^1-G-\{G,H\}\right|\right|\right|_{j_0,N,r-\sigma}\leq\left(\frac{e}{\sigma}\right)^2\left|\left|\left|H\right|\right|\right|_{j_0,N,r}^2\cdot
\left|\left|\left|G\right|\right|\right|_{j_0,N,r}.
\end{align}
\end{rem}

Let
\begin{align}\label{031601}
\epsilon=\epsilon_1+\epsilon_2,
\end{align}
and  introduce the non-resonant conditions.

\begin{defn}(\textbf{Non-resonant condition})
Given $\epsilon>0,$ $\alpha\in (0,1/100)$ and $N\in\mathbb{N}$, we say that the frequency $v=(v_j)_{j\in\mathbb{Z}}$ is $(\epsilon,\alpha,N)$-nonresonant if for any $0\neq k\in\mathbb{Z}^{\mathbb{Z}}$,
\begin{align}\label{122106}
\left|\sum_{j\in\mathbb{Z}}k_jv_j\right|\geq\frac{\epsilon^{\alpha}}{N\Delta^2(k)|k|^{\Delta(k)+1}}.
\end{align}
\end{defn}

\section{Analysis and Estimates of the Symplectic Transformations}

We now construct the symplectic transformation $\Gamma$ by a finite step induction.

At the first step, i.e., $s=1$ (in view of (\ref{010703}))
\begin{align}\label{011410}
H_1=H=\frac{1}2\left(\sum_{j\in\mathbb{Z}}v_j|q_j|^2+\epsilon_1\sum_{j\in\mathbb{Z}}\left(\bar q_jq_{j+1}+q_j\bar q_{j+1}\right)+\frac12\epsilon_2\sum_{j\in\mathbb{Z}}|q_j|^4\right),
\end{align}
which can be rewritten as
\begin{align*}
\nonumber H_1&=D_1+Z_1+R_1\\
&=\frac12\sum_{j\in\mathbb{Z}}v_{1j}|q_j|^2
+\sum_{n\in\mathbb{N}^{\mathbb{Z}}\times\mathbb{N}^{\mathbb{Z}}}Z_1(n)\prod_{{\rm supp}\ n}\left|q_j^{n_j}\right|^2
+\sum_{n\in\mathbb{N}^{\mathbb{Z}}\times\mathbb{N}^{\mathbb{Z}}}R_1(n)\prod_{{\rm supp}\ n}q_j^{n_j}\bar q_j^{n_j'},
\end{align*}
where
\begin{align*}
v_{1j}=v_j, \quad  Z_1(n)=\frac{\epsilon_2}4,\quad R_1(n)=\frac{\epsilon_1}2.
\end{align*}
From (\ref{031601}), we see that
\begin{align}\label{122502}
\left\|H_1-D_1\right\|_{j_0,N,r}\leq 10Nr^3\epsilon
\end{align}
and
\begin{align*}
\left\|H_1-D_1\right\|_{j_0,N,r}^{\mathcal{L}}=0,
\end{align*}
which implies
\begin{align*}
\left|\left|\left|H_1-D_1\right|\right|\right|_{j_0,N,r}\leq 10Nr^3\epsilon.
\end{align*}

\subsection{One step of Birkhoff normal form}
Let
\begin{align}\label{011601}
N_s=N-20(s-1),\  s\geq 1.
\end{align}

\begin{lem}\label{122501}
Let $v_1=(v_{1j})_{j\in\mathbb{Z}}$ satisfy the $(\epsilon,\alpha,N)$-nonresonant conditions \eqref{122106}. Assume $0<\sigma<r/2$ and
\begin{align}\label{052601}
\frac{2^6e}{\sigma}\cdot 10N^3r^3\epsilon^{1-2\alpha}\leq \frac12.
\end{align}
Then there exists a change of variables $\Gamma_1:=X_{F_1}^1$ such that
\begin{align*}
\nonumber H_{2}&=H_1\circ X_{F_1}^1=D_2+Z_2+R_2\\
&=\frac12\sum_{j\in\mathbb{Z}}v_{2j}|q_j|^2
+\sum_{n\in\mathbb{N}^{\mathbb{Z}}\times\mathbb{N}^{\mathbb{Z}}\atop n\in\mathcal{N},|n|\geq 4}Z_{2}(n)\prod_{{\rm supp}\ n}\left|q_j^{n_j}\right|^2
 +\sum_{n\in\mathbb{N}^{\mathbb{Z}}\times\mathbb{N}^{\mathbb{Z}}}R_{2}(n)\prod_{{\rm supp}\ n}q_j^{n_j}\bar q_j^{n_j'}.
\end{align*}
Moreover, one has
\begin{align}
\label{122507}\left|\left|\left|F_1\right|\right|\right|_{j_0,N,r}&\leq {2^{6}}\cdot 10N^3r^3\epsilon^{1-2\alpha},\\
\label{122601}\left|\left|\left|Z_2\right|\right|\right|_{j_0,N,r-\sigma}&\leq 10Nr^3\epsilon\left(\sum_{i=0}^{1}2^{-i}\right),\\
\label{122603}\left|\left|\left|R_2\right|\right|\right|_{j_0,N,r-\sigma}&\leq 10Nr^3\epsilon\left(\sum_{i=0}^{1}2^{-i}\right),
\end{align}
and
\begin{align}\label{011403}
\left|\left|\left|\mathcal{R}_2\right|\right|\right|_{j_0,N,r-\sigma}\leq 10Nr^3\epsilon\left( \frac{2^6e}{\sigma}\cdot 10N^3r^3\epsilon^{1-2\alpha}\right),
\end{align}
where
\begin{align}\label{122604}
\mathcal{R}_2=\sum_{n\in\mathbb{N}^{\mathbb{Z}}\times\mathbb{N}^{\mathbb{Z}}}R_{2}(n)\prod_{{\rm supp}\ n\cap A(j_0,N_3)\neq\emptyset}q_j^{n_j}\bar q_j^{n_j'}.
\end{align}
Furthermore, for any $A\geq 3$ the following estimate holds
\begin{align}\label{122605}
\nonumber&\left|\left|\left|\sum_{\Delta(n)+|n|=A}\left(|Z_2(n)|+|R_2(n)|\right)\prod_{{\rm supp}\ n}q_j^{n_j}\bar q_j^{n_j'}\right|\right|\right|_{j_0,N,r-\sigma}\\
\leq& 10Nr^3\epsilon\left( \frac{2^6e}{\sigma}\cdot 10N^3r^3\epsilon^{1-2\alpha}\right)^{A-3}.
\end{align}
\end{lem}

\begin{proof}
By the Birkhoff normal form theory, one knows that $F_1$ satisfies the homological equation
\begin{align}\label{hom1}
L_{v_1} F_1=\mathcal{R}_1,
\end{align}
where the \textit{Lie derivative} operator is defined by
\begin{align*}
L_{v_1}:\;H\mapsto L_{v_1}H:=\mathbf{i}\sum_{n\in\mathbb{N}^{\mathbb{Z}}\times\mathbb{N}^{\mathbb{Z}}}\left(\sum_{j\in\mathbb{Z}}(n_j-n_j')v_{1j}\right)H(n)\prod_{{\rm supp}\ n}q_j^{n_j}\bar q_j^{n_j'}
\end{align*}
and
\begin{align*}
\mathcal{R}_1=\sum_{n\in\mathbb{N}^{\mathbb{Z}}\times\mathbb{N}^{\mathbb{Z}}}R_1(n)\prod_{{\rm supp}\ n\cap A(j_0,N_2)\neq \emptyset}q_j^{n_j}\bar q_j^{n_j'}.
\end{align*}
Unless $n\in\mathcal{N}$ (see (\ref{030301})), one has
\begin{align*}
F_1(n)=\frac{R_1(n)}{\sum_{j\in\mathbb{Z}}(n_j-n_j')v_{1j}}.
\end{align*}
Note that frequency $v_1$ satisfies the nonresonant conditions (\ref{122106}).  Then we have
\begin{align}\label{122503} |F_1(n)|\leq {|R_1(n)|}\cdot\left({\epsilon^{-\alpha}}\cdot{N}\cdot \Delta^{2}(n)\cdot|n|^{\Delta(n)+1}\right).\end{align}
Noting that $|n|\leq 2\  {\rm and}\  \Delta(n)\leq 1,$
then
\begin{eqnarray}\label{042003}
\left\|F_1\right\|_{j_0,N,r}&\leq& \left\|R_1\right\|_{j_0,N,r}\cdot \epsilon^{-\alpha}\cdot N\cdot 2^2\leq 2^2\cdot 10N^2r^3\epsilon^{1-\alpha},
\end{eqnarray}where the last inequality is based on (\ref{122502}).
On the other hand, for any $\widetilde j\in\mathbb{Z}$ we have
\begin{align*}
\partial_{v_{\widetilde j}}F_1(n)=\frac{\partial_{v_{\widetilde j}}R_1(n)}{\sum_{j\in\mathbb{Z}}(n_j-n_j')v_{1j}}-\frac {R_1(n)} {\left(\sum_{j\in\mathbb{Z}}(n_j-n_j')v_{1j}\right)^2}\cdot\partial_{v_{\widetilde j}}\left(\sum_{j\in\mathbb{Z}}(n_j-n_j')v_{1j}\right).
\end{align*}
Then following the proof of (\ref{042003}), one has
\begin{align}\label{042004}
\left\|F_1\right\|^{\mathcal{L}}_{j_0,N,r}\leq 2^5\cdot 10N^3r^3\epsilon^{1-2\alpha}.
\end{align}
Then we finish the proof of (\ref{122507}) by using (\ref{042003}) and (\ref{042004}).

Using Taylor's formula yields
\begin{align*}
\nonumber H_{2}&:=H_1\circ X_{F_1}^1=D_1+Z_1\\
\nonumber &\ +\{D_1,F_1\}+R_1+\left(X_{F_1}^1-\textbf{id}-\{\cdot,F_1\}\right)D_1
+\left(X_{F_1}^1-\textbf{id}\right)(Z_1+R_1)\\
\nonumber&=D_2+Z_2+R_2\\
&:=\frac12\sum_{j\in\mathbb{Z}}v_{2j}|q_j|^2
+\sum_{n\in\mathbb{N}^{\mathbb{Z}}\times\mathbb{N}^{\mathbb{Z}}\atop n\in\mathcal{N},|n|\geq 4}Z_{2}(n)\prod_{{\rm supp}\ n}\left|q_j^{n_j}\right|^2
+\sum_{n\in\mathbb{N}^{\mathbb{Z}}\times\mathbb{N}^{\mathbb{Z}}}R_{2}(n)\prod_{{\rm supp}\ n}q_j^{n_j}\bar q_j^{n_j'},
\end{align*}
where by \eqref{hom1},
\begin{align}
\nonumber R_2 &=(R_1-\mathcal{R}_1)+\left(X_{F_1}^1-\textbf{id}-\{\cdot,F_1\}\right)D_1+\left(X_{F_1}^1-\textbf{id}\right)(Z_1+R_1)\\
\nonumber &=\sum_{n\in\mathbb{N}^{\mathbb{Z}}\times\mathbb{N}^{\mathbb{Z}}}R_{2}(n)\prod_{{\rm supp}\ n}q_j^{n_j}\bar q_j^{n_j'},
\end{align}
and
\begin{align*}
&\left(X_{F_1}^1-\textbf{id}-\{\cdot,F_1\}\right)D_1:=D_1\circ X_{F_1}^1-D_1-\{D_1,F_1\},\\
&\left(X_{F_1}^1-\textbf{id}\right)(Z_1+R_1):=(Z_1+R_1)\circ X_{F_1}^1-(Z_1+R_1).
\end{align*}
In the first step, we have $v_{2}=v_{1}$ and $Z_2=Z_1,$ which implies $D_2=D_1$ and $Z_2=Z_1$. Hence, the estimate (\ref{122601}) holds true.

Write
\begin{align*}
R_2=\mathcal{R}_2+(R_2-\mathcal{R}_2),
\end{align*}
where $\mathcal{R}_2$ is defined by (\ref{122604}).
By (\ref{9}) and (\ref{10}) in Remark \ref{092403} and  \eqref{hom1}, for any $0<\sigma<r/2$ one has
\begin{align*}
\left|\left|\left|\mathcal{R}_2\right|\right|\right|_{j_0,N,r-\sigma}&\leq\nonumber \left(\frac{e}{\sigma}\right)\cdot\left|\left|\left|F_1\right|\right|\right|_{j_0,N,r}\cdot\left|\left|\left|H_1\right|\right|\right|_{j_0,N,r}\\
&\leq10Nr^3\epsilon\left(\frac{2^6e}{\sigma}\cdot 10N^3r^3\epsilon^{1-2\alpha}\right),
\end{align*}
where the last inequality follows from (\ref{122502}) and (\ref{122507}). This finishes the proof of (\ref{011403}).
Similarly, we have
\begin{align*}
\left|\left|\left|R_2-\mathcal{R}_2\right|\right|\right|_{j_0,N,r-\sigma}&\leq 10Nr^3\epsilon+10Nr^3\epsilon\left(\frac{2^6e}{\sigma}\cdot 10N^3r^3\epsilon^{1-2\alpha}\right)\\
&\leq 10Nr^3\epsilon\left(\sum_{i=0}^12^{-i}\right),
\end{align*}
where the last inequality is based on  (\ref{052601}).

Finally, the estimate  (\ref{122605}) follows from (\ref{122502}) and (\ref{122507}) by using by induction about $A$. Precisely,
the term in $R_2$ comes from $\frac1{j!}Z_1^{(j)}$ and $\frac1{ {j}!}R_1^{({j})}$ for some $j\in \mathbb{N}$, where
$Z_1^{(j)}=\left\{Z_1^{(j-1)},H\right\}$, $Z_1^{(0)}=Z_1$, $R_1^{(j)}=\left\{R_1^{(j-1)},H\right\}$ and  $R_1^{(0)}=R_1$. Following the proof of (\ref{122509}) and noting that
$\Delta(l)\leq \Delta(n)+\Delta(m)$ and $|l|\leq |n|+|m|-2,$ we finish the proof of (\ref{122605}).
\end{proof}

\subsection{Iterative Lemma}
\begin{lem}\label{62}
For $s\in\mathbb{N}$ and $1\leq s\leq \sqrt{N}-1$, consider the Hamiltonian $H_s(q,\bar q)$ of the form
\begin{align*}
\nonumber H_s&=D_s+Z_s+R_s\\
&=\frac12\sum_{j\in\mathbb{Z}}v_{sj}|q_j|^2+\sum_{n\in\mathbb{N}^{\mathbb{Z}}\times\mathbb{N}^{\mathbb{Z}}\atop n\in\mathcal{N},|n|\geq4}Z_s(n)\prod_{{\rm supp}\ n}\left|q_j^{n_j}\right|^2
+\sum_{n\in\mathbb{N}^{\mathbb{Z}}\times\mathbb{N}^{\mathbb{Z}}}R_s(n)\prod_{{\rm supp}\ n}q_j^{n_j}\bar q_j^{n_j'}.
\end{align*}
Let $v_s=(v_{sj})_{j\in\mathbb{Z}}$ satisfy the $(\epsilon,\alpha,N)$-nonresonant condition \eqref{122106}. Assume that $0<\sigma<r/2$ and
\begin{align}
&\label{011412}\frac{\left(10(s+1)\right)^{10(s+1)}\cdot 2^6e}{\sigma}\cdot N^{3(s+1)}r^3\epsilon^{1-2\alpha}\leq \frac12,\\
&\label{122601..}\left|\left|\left|Z_s\right|\right|\right|_{j_0,N,r-(s-1)\sigma}\leq 10Nr^3\epsilon\left(\sum_{i=0}^{s-1}2^{-i}\right),\\
&\label{122603..}\left|\left|\left|R_s\right|\right|\right|_{j_0,N,r-(s-1)\sigma}\leq 10Nr^3\epsilon\left(\sum_{i=0}^{s-1}2^{-i}\right),\\
&\label{122602..}\left|\left|\left|\mathcal{R}_s\right|\right|\right|_{j_0,N,r-(s-1)\sigma}\leq 10Nr^3\epsilon\left( \frac{(10s)^{10s}\cdot 2^6e}{\sigma}\cdot N^{3s}r^3\epsilon^{1-2\alpha}\right)^{s-1},
\end{align}
where
\begin{align}\label{122604.}
\mathcal{R}_s=\sum_{n\in\mathbb{N}^{\mathbb{Z}}\times\mathbb{N}^{\mathbb{Z}}}R_{s}(n)\prod_{{\rm supp}\ n\cap A(j_0,N_{s+1})\neq\emptyset}q_j^{n_j}\bar q_j^{n_j'}.
\end{align}
Furthermore, assume for any $A\geq 3$ the following holds
\begin{align}\label{122605.}
\nonumber&\left|\left|\left|\sum_{\Delta(n)+|n|=A}\left(|Z_s(n)|+|R_s(n)|\right)\prod_{{\rm supp}\ n}q_j^{n_j}\bar q_j^{n_j'}\right|\right|\right|_{j_0,N,r-(s-1)\sigma}\\
\leq& 10Nr^3\epsilon\left( \frac{(10s)^{10s}\cdot 2^6e}{\sigma}\cdot N^{3s}r^3\epsilon^{1-2\alpha}\right)^{A-3}.
\end{align}

Then there exists a change of variables $\Phi_s:=X_{F_s}^1$
\begin{align*}
\nonumber H_{s+1}&=H_s\circ X_{F_s}^1\\
&=\frac12\sum_{j\in\mathbb{Z}}v_{(s+1)j}|q_j|^2
+\sum_{n\in\mathbb{N}^{\mathbb{Z}}\times\mathbb{N}^{\mathbb{Z}}\atop n\in\mathcal{N},|n|\geq 4}Z_{s+1}(n)\prod_{{\rm supp}\ n}\left|q_j^{n_j}\right|^2\\
&\ +\sum_{n\in\mathbb{N}^{\mathbb{Z}}\times\mathbb{N}^{\mathbb{Z}}}R_{s+1}(n)\prod_{{\rm supp}\ n}q_j^{n_j}\bar q_j^{n_j'}.
\end{align*}
Moreover, one has
\begin{align}
&\label{122610} \left|\left|\left|F_s\right|\right|\right|_{j_0,N,r-(s-1)\sigma}\leq \left(\frac{\sigma}{e}\right)\left( \frac{(10s)^{10s}\cdot 2^6e}{\sigma}\cdot N^{3s}r^3\epsilon^{1-2\alpha}\right)^{s},\\
&\label{122616}\left|\left|\left|Z_{s+1}\right|\right|\right|_{j_0,N,r-s\sigma}\leq 10Nr^3\epsilon\left(\sum_{i=0}^{s}2^{-i}\right),\\
&\label{122615}\left|\left|\left|R_{s+1}\right|\right|\right|_{j_0,N,r-s\sigma}\leq 10Nr^3\epsilon\left(\sum_{i=0}^{s}2^{-i}\right),\\
&\label{122611}\left|\left|\left|\mathcal{R}_{s+1}\right|\right|\right|_{j_0,N,r-s\sigma}\leq10Nr^3\epsilon\cdot \left(\frac{(10(s+1))^{10(s+1)}\cdot 2^6e}{\sigma}\cdot N^{3(s+1)}r^3\epsilon^{1-2\alpha}\right)^{s},
\end{align}
where
\begin{align*}
\mathcal{R}_{s+1}=\sum_{n\in\mathbb{N}^{\mathbb{Z}}\times\mathbb{N}^{\mathbb{Z}}}R_{s+1}(n)\prod_{{\rm supp}\ n\cap A(j_0,N_{s+2})\neq \emptyset}q_j^{n_j}\bar q_j^{n_j'}.
\end{align*}
Moreover, we have
\begin{align}
\nonumber&\left|\left|\left|\sum_{\Delta(n)+|n|=A}\left(|Z_{s+1}(n)|+|R_{s+1}(n)|\right)\prod_{{\rm supp}\ n}q_j^{n_j}\bar q_j^{n_j'}\right|\right|\right|_{j_0,N,r-s\sigma}\\
\label{122617}\leq& 10Nr^3\epsilon\left( \frac{(10(s+1))^{10(s+1)}\cdot 2^6e}{\sigma}\cdot N^{3(s+1)}r^3\epsilon^{1-2\alpha}\right)^{A-3}.
\end{align}

\end{lem}
\begin{proof}
As done before, we know that $F_s$ will satisfy the homological equation
\begin{align*}
L_{v_s} F_s=\widetilde{\mathcal R}_s,
\end{align*}
where
\begin{align}\label{022501}
\widetilde{\mathcal{R}}_s(q,\bar q):=\sum_{n\in\mathbb{N}^{\mathbb{Z}}\times\mathbb{N}^{\mathbb{Z}}}R_{s}(n)\prod_{{\rm supp}\ n\cap A(j_0,N_{s+1})\neq \emptyset\atop
\Delta (n)+|n|\leq s+2}q_j^{n_j}\bar q_j^{n_j'}.
\end{align}
By the direct computations, one has
\begin{align*}
F_s(n)=\frac{R_s(n)}{\sum_{j\in\mathbb{Z}}(n_j-n_j')v_{sj}},
\end{align*}
unless $n\in\mathcal{N}$.
Since the frequency $v_s$ satisfies the $(\epsilon,\alpha,N)$-nonresonant condition (\ref{122106}), we get
\begin{align}\label{111901..}
|F_s(n)|\leq {|R_{s}(n)|}\cdot\epsilon^{-\alpha}\cdot{N}\cdot \Delta^2(n)\cdot|n|^{\Delta(n)+1}.\end{align}
In view of (\ref{122602..})
we have
\begin{align}
\nonumber&\left\|F_s\right\|_{j_0,N,r-(s-1)\sigma}\\
\nonumber&\leq10Nr^3\epsilon\left( \frac{(10s)^{10s}\cdot 2^6e}{\sigma}\cdot N^{3s}r^3\epsilon^{1-2\alpha}\right)^{s-1}\cdot\epsilon^{-\alpha}\cdot{N}\cdot \Delta^2(n)\cdot|n|^{\Delta(n)+1}\\
\label{042005}&\leq\left(10N^2r^3\epsilon^{1-\alpha}\left(4s\right)^{4s}\right)\left( \frac{(10s)^{10s}\cdot 2^6e}{\sigma}\cdot N^{3s}r^3\epsilon^{1-2\alpha}\right)^{s-1},
\end{align}
where the last inequality is based on $\Delta(n)+|n|\leq s+2.$
Similarly, one has
\begin{align}\label{042006}
\left\|F_s\right\|_{j_0,N,r-(s-1)\sigma}^{\mathcal{L}}\leq \left(10N^3r^3\epsilon^{1-2\alpha}\left(4s\right)^{8s}\right)\left( \frac{(10s)^{10s}\cdot 2^6e}{\sigma}\cdot N^{3s}r^3\epsilon^{1-2\alpha}\right)^{s-1}.
\end{align}
In view of (\ref{042005}) and (\ref{042006}), one has
 \begin{align*}
 \left|\left|\left|F_s\right|\right|\right|_{j_0,N,r-(s-1)\sigma}&\leq \left(20N^3r^3\epsilon^{1-2\alpha}\left(4s\right)^{8s}\right)\left( \frac{(10s)^{10s}\cdot 2^6e}{\sigma}\cdot N^{3s}r^3\epsilon^{1-2\alpha}\right)^{s-1}\\
 &\leq\left(\frac{\sigma}{e}\right)\left( \frac{(10s)^{10s}\cdot 2^6e}{\sigma}\cdot N^{3s}r^3\epsilon^{1-2\alpha}\right)^{s},
 \end{align*}
 which finishes the proof of (\ref{122610}).

Using Taylor's formula again shows
\begin{align*}
H_{s+1}&:=H_s\circ X_{F_s}^1\\
&=D_s+\{D_s,F_s\}+Z_s+R_s+\left(X_{F_s}^1-\textbf{id}-\{\cdot,F_s\}\right)D_s
+\left(X_{F_s}^1-\textbf{id}\right)(Z_s+R_s)\\
&=D_{s+1}+Z_{s+1}+R_{s+1}\\
&=\frac12\sum_{j\in\mathbb{Z}}v_{(s+1)j}|q_j|^2+\sum_{n\in\mathbb{N}^{\mathbb{Z}}\times\mathbb{N}^{\mathbb{Z}}\atop
n\in\mathcal{N},|n|\geq 4}Z_{s+1}(n)\prod_{{\rm supp}\ n}\left|q_j^{n_j}\right|^2\\
&\ +\sum_{n\in\mathbb{N}^{\mathbb{Z}}\times\mathbb{N}^{\mathbb{Z}}}R_{s+1}(n)\prod_{{\rm supp}\ n}q_j^{n_j}\bar q_j^{n_j'}.
\end{align*}
Precisely, let
\begin{align*}
G_{s+1}&=\{D_s,F_s\}+R_s+\left(X_{F_s}^1-\textbf{id}-\{\cdot,F_s\}\right)D_s
+\left(X_{F_s}^1-\textbf{id}\right)(Z_s+R_s)\\
&=\sum_{n\in\mathbb{N}^{\mathbb{Z}}\times\mathbb{N}^{\mathbb{Z}}}G_{s+1}(n)\prod_{{\rm supp}\ n}q_j^{n_j}\bar q_j^{n_j'},
\end{align*}
and then one has
\begin{align}
\label{031104}& D_{s+1}=D_s+\sum_{n\in\mathbb{N}^{\mathbb{Z}}\times\mathbb{N}^{\mathbb{Z}}\atop
n\in\mathcal{N},|n|=2}G_{s+1}(n)\prod_{{\rm supp}\ n}q_j^{n_j}\bar q_j^{n_j'},\\
\nonumber &Z_{s+1}=Z_s+\sum_{n\in\mathbb{N}^{\mathbb{Z}}\times\mathbb{N}^{\mathbb{Z}}\atop
n\in\mathcal{N},|n|\geq 4}G_{s+1}(n)\prod_{{\rm supp}\ n}q_j^{n_j}\bar q_j^{n_j'},\\
\nonumber &R_{s+1}=\sum_{n\in\mathbb{N}^{\mathbb{Z}}\times\mathbb{N}^{\mathbb{Z}}\atop
n\notin\mathcal{N}}G_{s+1}(n)\prod_{{\rm supp}\ n}q_j^{n_j}\bar q_j^{n_j'}.
\end{align}
Write
\begin{align*}
R_{s+1}=\mathcal{R}_{s+1}+(R_{s+1}-\mathcal{R}_{s+1}),
\end{align*}
where
\begin{align*}
{\mathcal{R}}_{s+1}=\sum_{n\in\mathbb{N}^{\mathbb{Z}}\times\mathbb{N}^{\mathbb{Z}}}R_{s+1}(n)\prod_{{\rm supp}\ n\cap A(j_0,N_{s+2})\neq\emptyset}q_j^{n_j}\bar q_j^{n_j'}.
\end{align*}

By (\ref{9}) and (\ref{10}) in Remark \ref{092403} and (\ref{122610}), one has
\begin{align*}
\left|\left|\left|\mathcal{R}_{s+1}\right|\right|\right|_{j_0,N,r-s\sigma}&\leq\left(\frac{e}{\sigma}\right)\left|\left|\left|F_s\right|\right|\right|_{j_0,N,r-(s-1)\sigma}
\left|\left|\left|Z_s+R_s\right|\right|\right|_{j_0,N,r-(s-1)\sigma}\\
&\leq\left(\frac{(10s)^{10s}\cdot 2^6e}{\sigma}\cdot N^{3s}r^3\epsilon^{1-2\alpha}\right)^{s}\cdot10Nr^3\epsilon\left(\sum_{i=0}^{s-1}2^{-i}\right)\\
&\leq10Nr^3\epsilon\cdot \left(\frac{(10(s+1))^{10(s+1)}\cdot 2^{6}e}{\sigma}\cdot N^{3(s+1)}r^3\epsilon^{1-2\alpha}\right)^{s},
\end{align*}
which finishes the proof of (\ref{122611}).
 Similarly, we have
\begin{equation*}
\left|\left|\left|R_{s+1}-\mathcal{R}_{s+1}\right|\right|\right|_{j_0,N,r-s\sigma}\leq 10Nr^3\epsilon\left(\sum_{i=0}^{s}2^{-i}\right),
\end{equation*}
where the last inequality is based on  (\ref{011412}). This finishes the proof of (\ref{122615}).
Similarly,  one has
\begin{align*}
\left|\left|\left|Z_{s+1}\right|\right|\right|_{j_0,N,r-s\sigma}\leq 10Nr^3\epsilon\left(\sum_{i=0}^{s}2^{-i}\right),
\end{align*}
which finishes the proof of (\ref{122616}).

Finally, the estimate (\ref{122617}) follows from the proof of (\ref{122605}).


\end{proof}
\subsection{The Birkhoff Normal Form Theorem}
In this subsection, we will establish the Birkhoff normal form theorem. Fix
\begin{align}\label{051802}
N=\left|\frac{\ln \epsilon}{200\ln\left|\ln\epsilon\right|}\right|^2.
\end{align}
We begin with a key lemma  in dealing with the nonresonant condition (\ref{122106}). Denote by ${\rm mes}(\cdot)$ the standard product measure on $[0,1]^{\Z}$.
{\begin{lem}\label{lem101}
Fix $\alpha\in(0,1/100)$ and $j_0\in\N$. Then for $0<\epsilon<\epsilon(\alpha)\ll1$,  there exists some  $\mathfrak{R}(j_0)\subset [0,1]^{\Z}$  satisfying
\begin{align*}
{\rm mes}(\mathfrak{R}(j_0))\leq \epsilon^{\alpha/2}
\end{align*}
such that the following holds:  if $v_1=(v_{1j})_{j\in\Z}\in[0,1]^{\Z}\setminus\mathfrak{R}(j_0)$, then all $v_s=(v_{sj})_{j\in\Z}$  with $1\leq s\leq M<\sqrt{N}-1$ will satisfy the nonresonant condition (\ref{122106}), where $v_s$ ($2\leq s\leq M$) are inductively defined in the Iterative Lemma (i.e., Lemma \ref{62}).
\end{lem}
\begin{rem}
Let us comment on the definition and nonresonant properties of $v_s$ first. Assume $v_1=(v_{1j})_{j\in\Z}$ satisfies the nonresonant condition (\ref{122106}). Then using Lemma \ref{62} yields a modulated frequency $v_2=(v_{2j})_{j\in\Z}$ which depends on $v_1$. At this stage, $v_2$ may not satisfy the nonresonant condition (\ref{122106}). To propagate the Iterative Lemma, one can make further restrictions on $v_1$ so that $v_2$ satisfies (\ref{122106}). Repeating this procedure and removing more $v_1$ can ensure all $v_s$ ($1\leq s\leq M<\sqrt{N}-1$) satisfy (\ref{122106}).  The detailed proof is postponed to the next section.
\end{rem}}

Let $v=v_1=(v_{1j})_{j\in\Z}\in[0,1]^{\Z}\setminus\mathfrak{R}(j_0)$.  Then applying the Iterative Lemma gives
\begin{thm}[\textbf{Birkhoff Normal Form}]\label{011501} Consider the Hamiltonian \eqref{011410} and assume $v=v_1=(v_{1j})_{j\in\Z}\in[0,1]^{\Z}\setminus\mathfrak{R}(j_0)$. Given any $r>2$, then there exists an $\epsilon^*(r,\alpha)>0$ such that, for any $0<\epsilon<\epsilon^*(r,\alpha)$ and any $M\in\mathbb{N}$ with $M<\sqrt{N}-1$, there exists a symplectic transformation $\Gamma=\Gamma_1\circ\cdots\circ\Gamma_M$
such that
\begin{align*}
\widetilde H&=H_1\circ \Gamma=\widetilde{D}+\widetilde{Z}+\widetilde{R}\\
&=\frac12\sum_{j\in\mathbb{Z}}\widetilde v_{j}|q_j|^2+\sum_{n\in\mathbb{N}^{\mathbb{Z}}\times\mathbb{N}^{\mathbb{Z}}\atop n\in\mathcal{N},|n|\geq4}\widetilde Z(n)\prod_{{\rm supp}\ n}q_j^{n_j}\bar q_j^{n_j'}\\
&\ \ \ \ +\sum_{n\in\mathbb{N}^{\mathbb{Z}}\times\mathbb{N}^{\mathbb{Z}}}\widetilde R(n)\prod_{{\rm supp}\ n}q_j^{n_j}\bar q_j^{n_j'},
\end{align*}
where
\begin{align}
&\label{030401} \left|\left|\left|\widetilde {Z}\right|\right|\right|_{j_0,N,r/2}\leq 20Nr^3\epsilon,\\
&\label{030402} \left|\left|\left|\widetilde R\right|\right|\right|_{j_0,N,r/2}\leq 20Nr^3\epsilon,
\end{align}
and
\begin{align}
\label{030403}&\left|\left|\left|\widetilde{\mathcal{R}}\right|\right|\right|_{j_0,N,r/2}\leq10Nr^3\epsilon\cdot \left({(10(M+1))^{10(M+1)}\cdot 2^6e}\cdot N^{3(M+1)+1}r^2\epsilon^{1-2\alpha}\right)^{M},
\end{align}
with
\begin{align}
\label{042301}&\widetilde{\mathcal{R}}=\sum_{n\in\mathbb{N}^{\mathbb{Z}}\times\mathbb{N}^{\mathbb{Z}}}\widetilde{R}(n)\prod_{{\rm supp}\ n\cap A(j_0,N/2)\neq \emptyset}q_j^{n_j}\bar q_j^{n_j'}.
\end{align}
Furthermore, for any $A\geq 3$ the following estimate holds
\begin{align}
\nonumber&\left|\left|\left|\sum_{\Delta(n)+|n|=A}\left(|\widetilde{Z}(n)|+|\widetilde{R}(n)|\right)\prod_{{\rm supp}\ n}q_j^{n_j}\bar q_j^{n_j'}\right|\right|\right|_{j_0,N,r/2} \\
\label{042302}\leq& 10Nr^3\epsilon\left( {(10(M+1))^{10(M+1)}\cdot 2^6e}\cdot N^{3(M+1)+1}r^2\epsilon^{1-2\alpha}\right)^{A-3}.
\end{align}
\end{thm}

\begin{proof}
First of all, note that the Hamiltonian (\ref{011410}) satisfies all assumptions (\ref{122601..})--(\ref{122605.}) for $s=1$, which follows from (\ref{122502}).

Secondly, for given $r>2$ and $0<\alpha<1/100$, we take $\epsilon^*=\epsilon^{*}(r,\alpha)>0$ such that
\begin{align*}
\left(10N^*\right)^{20{N^*}}\cdot {2^{10}er^2}\cdot (\epsilon^*)^{1-2\alpha}\leq \frac12,
\end{align*}
where
\begin{align*}
N^{*}=\left|\frac{\ln \epsilon^*}{200\ln\left|\ln\epsilon^*\right|}\right|^2.
\end{align*}
Then for any $0<\epsilon<\epsilon^*(r,\alpha)$ and any $1\leq s\leq M$, the assumption (\ref{011412}) holds with $\sigma=\frac{r}{2N}$. Moreover, one has
\begin{equation*}
r-s\sigma\geq r-M\cdot \frac{r}{2N}\geq r/2.
\end{equation*}
In view of (\ref{011601}) and for any $1\leq s\leq M$, one has
\begin{equation*}
N_{s+1}=N-20s\geq N-20M=N-20\sqrt{N}\geq \frac{N}2,
\end{equation*}
which implies
\begin{equation*}
\left[j_0-\frac{N}2,j_0+\frac{N}2\right]\subset A(j_0,N_s)\subset\left[j_0-{N},j_0+{N}\right].
\end{equation*}

Finally, it follows from Lemma \ref{lem101}  that all $v_s$ ($1\leq s\leq M$) satisfy the  $(\epsilon,\alpha,N)$-nonresonant condition (\ref{122106}). Then by using Iterative Lemma, one can find a symplectic transformation $\Gamma=\Gamma_1\circ\cdots\circ\Gamma_M$ such that
\begin{align*}
\widetilde H:=H_{M+1}=H_1\circ \Gamma,
 \end{align*}which satisfies (\ref{030401}), (\ref{030402}) and (\ref{030403}).
\end{proof}

\section{Estimate on the measure}
In this section, we complete the proof of Lemma \ref{lem101}

\begin{proof}[Proof of Lemma \ref{lem101}]
Given $N>0, j_0\in\N, n\in\N^{\Z}$ and $1\leq s\leq \left[\sqrt{N}\right]-1$, define the resonant set $\mathfrak{R}_s(n)$ by
\begin{align*}
\mathfrak{R}_s(n)=\left\{v_1=(v_{1j})_{j\in\mathbb{Z}}\in[0,1]^{\Z}:\ \left|\sum_{j\in\mathbb{Z}}(n_j-n_j')v_{sj}\right|<\frac{\epsilon^{\alpha}}{N\Delta^2(n)|n|^{\Delta(n)+1}}\right\}.
\end{align*}
Let
\begin{align*}
\mathfrak{R}(j_0)=\bigcup_{s=1}^{\left[\sqrt{N}\right]-1} \mathfrak{R}_s(j_0),
\end{align*}
where $\mathfrak{R}_s(j_0)=\cup_{n}^{s(*)}\mathfrak{R}_{s}(n)$
and the union $\bigcup_{n}^{s(*)}$ is taken  for $n$ satisfying ${\rm supp}\ n\cap A(j_0,N_{s+1})\neq \emptyset$ and $\Delta (n)+|n|\leq s+2.$
Obviously, we have the counting bound£º
\begin{align}
\nonumber&\#\{n:\ {\rm supp}\ n\cap A(j_0,N_{s+1})\neq \emptyset,\  \Delta (n)=a, |n|=b\}\\
\label{count}&\ \ \ \leq C(b+N_{s+1})b^a,
\end{align}
where $C>0$ is some absolute constant.

It is easy to see that 
\begin{align}\label{rsp101}
{\rm mes}\left(\mathfrak{R}_s'(n)\right)\leq \frac{C\epsilon^{\alpha}}{N\Delta^2(n)|n|^{\Delta(n)+1}},
\end{align}
where
\begin{align*}
\mathfrak{R}_s'(n)=\left\{v_s:\ \left|\sum_{j\in\mathbb{Z}}(n_j-n_j')v_{sj}\right|<\frac{\epsilon^{\alpha}}{N\Delta^2(n)|n|^{\Delta(n)+1}}\right\}.
\end{align*}

Consider first the case $s=1$. Then we have by \eqref{count} and \eqref{rsp101}
\begin{align}
\nonumber \mbox{\rm mes} (\mathfrak{R}_1(j_0))&\leq \sum_{{\rm supp}\ n\cap A(j_0,N_{2})\neq \emptyset\atop
\Delta (n)+|n|\leq 3}{\rm mes} (\mathfrak{R}_1(n))\\
\nonumber&\leq C\epsilon^{\alpha} \sum_{{\rm supp}\ n\cap A(j_0,N_{2})\neq \emptyset\atop
\Delta (n)+|n|\leq 3}\frac{1}{N\Delta^2(n)|n|^{\Delta(n)+1}}\\
\label{r1j0}
&\leq  C\frac{N_2}{N}\epsilon^{\alpha}.
\end{align}

For $2\leq s\leq M$, let $w^{(s)}=\left(w^{(s)}_j\right)_{j\in\mathbb{Z}}$ with $w^{(s)}_j=v_{sj}-v_{1j}$ and $W_s=\sum\limits_{j\in\mathbb{Z}}w^{(s)}_jq_j\bar q_j$. One sees that $v_{sj}-v_{1j}=0$ unless $\left|\left|j\right|-j_0\right|\leq N+1$. Moreover, in view of (\ref{122602..}), one has
\begin{align*}
\left|\left|\left|W_s\right|\right|\right|_{j_0,N,r-(s-1)\sigma}\leq 20Nr^{3}\epsilon.
\end{align*}
From Schur's test,
\begin{align}\label{042101}
\left|\left|\frac{\partial w^{(s)}}{\partial v_1}\right|\right|_{\ell^2\rightarrow\ell^2}\leq 40Nr^{3}\epsilon(s+2)\leq 40N^2r^3\epsilon,
\end{align}
as $\Delta(n)\leq s+2$ and $s\leq M<N$. Moreover, (\ref{042101}) implies that the frequency modulation map $v_1\rightarrow  v_s=v_1+w^{(s)}$ satisfies
\begin{align*}
e^{-1}\leq \left(1-40N^2r^3\epsilon\right)^{2N+2}\leq \left|\det \frac{\partial  v_s}{\partial v_1}\right|\leq \left(1+40N^2r^3\epsilon\right)^{2N+2}\leq e.
\end{align*}
Hence, one has by \eqref{rsp101} and \eqref{042101}
\begin{align*}
{\rm mes}(\mathfrak{R}_s(n))\leq e\cdot{\rm mes}(\mathfrak{R}_s'(n))\leq  \frac{C\epsilon^{\alpha}}{N\Delta^2(n)|n|^{\Delta(n)+1}}.
\end{align*}
Similar to the proof of \eqref{r1j0}, we have
\begin{align*}
{\rm mes}(\mathfrak{R}_s(j_0))\leq Cs^4\frac{N_{s+1}}{N}\epsilon^{\alpha}.
\end{align*}

Finally, by recalling \eqref{051802}, we obtain
\begin{align*}
{\rm mes}(\mathfrak{R}(j_0))&\leq\sum_{s=1}^{[\sqrt{N}]-1}{\rm mes}(\mathfrak{R}_s(j_0))\leq C|\log\epsilon|^5\epsilon^\alpha\leq \epsilon^{\alpha/2}.
\end{align*}
This  finishes the proof of Lemma \ref{lem101}.
\end{proof}

\section{Proof of main theorem}

Now we are in a position to complete the proof of Theorem \ref{main}.

\begin{proof}[Proof of Theorem \ref{main}]
In view of Theorem \ref{011501}, one obtains
the  $\widetilde H(\tilde q,\bar{\tilde q})$ in new coordinates. Then the new Hamiltonian equation is given by
\begin{align}\label{011503}
\textbf i \dot{\tilde q}=2\frac{\partial\widetilde{H}}{\partial \bar{\tilde q}}.
\end{align}
We get by using (\ref{011503}) that
\begin{align*}
\frac{d}{dt}\sum_{|j|>j_0}\left|\tilde q_j(t)\right|^2=&\left\{\sum_{|j|>j_0}\left|\tilde q_j(t)\right|^2,\widetilde D+\widetilde Z+\widetilde R\right\}\\
=&\left\{\sum_{|j|>j_0}\left|\tilde q_j(t)\right|^2,\widetilde R\right\}\\=&4\mbox{Im}  \sum_{|j|>j_0}\bar{\tilde q}_j(t)\frac{\partial\widetilde{R}}{\partial \bar{\tilde q}}\\
=&\sum_{n\in\mathbb{N}^{\mathbb{Z}}\times\mathbb{N}^{\mathbb{Z}}}
\widetilde{R}(n)\sum_{|j|>j_0}(n_j-n_j')\prod_{{\rm supp}\ n}{\tilde q_j}^{n_j}\bar {\tilde q}_j^{n_j'}.
\end{align*}
In view of (\ref{042301}), we decompose $\widetilde R$ into  three parts:
\begin{align*}
\widetilde R=\widetilde {{R}}^{(1)}+ \widetilde {{R}}^{(2)}+\widetilde {{R}}^{(3)},
\end{align*}
where
\begin{align}
\nonumber&\widetilde {{R}}^{(1)}=\widetilde{\mathcal{R}},\\
\nonumber&\widetilde {{R}}^{(2)}=\sum_{n\in\mathbb{N}^{\mathbb{Z}}\times\mathbb{N}^{\mathbb{Z}}}
\widetilde{R}(n)\sum_{|j|>j_0}(n_j-n_j')\prod_{{\rm supp}\ n\cap A(j_0,N/2)=\emptyset\atop \Delta(n)\geq M+4}{\tilde q_j}^{n_j}\bar {\tilde q}_j^{n_j'},\\
\label{042304}&\widetilde {{R}}^{(3)}=\sum_{n\in\mathbb{N}^{\mathbb{Z}}\times\mathbb{N}^{\mathbb{Z}}}
\widetilde{R}(n)\sum_{|j|>j_0}(n_j-n_j')\prod_{{\rm supp}\ n\cap A(j_0,N/2)=\emptyset\atop \Delta(n)\leq M+3}{\tilde q_j}^{n_j}\bar {\tilde q}_j^{n_j'}.
\end{align}
Using (\ref{030403}) and (\ref{042302}) implies
\begin{align*}
\left|\left|\left|\widetilde{{R}}^{(1)}+\widetilde{{R}}^{(2)}\right|\right|\right|_{j_0,N,r/2}\leq20Nr^3\epsilon\cdot \left({(10(M+1))^{10(M+1)}\cdot 2^6e}\cdot N^{3(M+1)+1}r^2\epsilon^{1-2\alpha}\right)^{M}.
\end{align*}
Take
\begin{align*}
M=\left[\sqrt{N}\right]-1\approx\left|\frac{\ln \epsilon}{200\ln\left|\ln\epsilon\right|}\right|.
\end{align*}
Then one has
\begin{align}\label{042303}
\left|\left|\left|\widetilde{{R}}^{(1)}+\widetilde{{R}}^{(2)}\right|\right|\right|_{j_0,N,r/2}\leq \epsilon\cdot \exp\left(-\frac{\left|\ln \epsilon\right|^2}{200\ln\left|\ln\epsilon\right|}\right),
\end{align}
where we use $0<\alpha<\frac1{100}$ and $\epsilon\ll1$.

Now consider the monomials in $\widetilde{{R}}^{(3)}$. Recalling that
\begin{align*}
\Delta(n)\leq M+3<2\sqrt{N},
\end{align*}
if ${\rm supp}\ n\cap A(j_0,N/2)= \emptyset$, then
\begin{align*}
\mbox{supp}\ n\subset\left(-\infty,-j_0\right)
\cup\left(j_0,\infty\right).
\end{align*}
Hence the terms in (\ref{042304}) satisfy
\begin{align}\label{042305}
\sum_{|j|>j_0}(n_j-n_j')=0
\end{align}
Using (\ref{042303}) and (\ref{042305}), one has
\begin{align*}
\frac{d}{dt}\sum_{|j|>j_0}\left|\tilde q_j(t)\right|^2\leq \epsilon\cdot \exp\left(\frac{\left|\ln \epsilon\right|^2}{200\ln\left|\ln\epsilon\right|}\right).
\end{align*}

Integrating in $t$, we obtain
 \begin{equation}\label{051801}
 \sum_{|j|>j_0}\left|\tilde q_j(t)\right|^2\leq \sum_{|j|>j_0}\left|\tilde q_j(0)\right|^2+\epsilon\cdot \exp\left(-\frac{\left|\ln \epsilon\right|^2}{200\ln\left|\ln\epsilon\right|}\right)t.
 \end{equation}

 Note that the symplectic transformation only acts on the $N$-neighborhood of $\pm j_0$.  We obtain
 \begin{align*}
 \sum_{|j|>j_0+N}|q_j(t)|^2\leq\sum_{|j|>j_0}|\tilde q_j(t)|^2,
 \end{align*}
 which together with (\ref{051801}) gives
 \begin{align*}
 \sum_{|j|>j_0+N}|q_j(t)|^2\leq \sum_{|j|>j_0}\left|\tilde q_j(0)\right|^2+\epsilon\cdot \exp\left(-\frac{\left|\ln \epsilon\right|^2}{200\ln\left|\ln\epsilon\right|}\right)t.
 \end{align*}

 On the other hand, the Hamiltonian preserves the $\ell^2$-norm. So we have
 \begin{align*}
 \sum_{|j|>j_0}|\tilde q_j(0)|^2=\sum_{j\in\mathbb{Z}}|q_j(0)|^2-\sum_{|j|\leq j_0}|\tilde q_j(0)|^2<\sum_{|j|\leq j_0-N}|q_j(0)|^2.
 \end{align*}
 Choosing $\bar{j}_0$ large enough and letting $j_0\in[\bar{j}_0,2\bar{j}_0 ]$ such that
 \begin{align*}
 \sum_{|j|>j_0-N}|q_j(0)|^2<\delta,
 \end{align*}
 then for
 \begin{align*}\label{051803}
 |t|\leq \delta\cdot \exp\left(\frac{\left|\ln \epsilon\right|^2}{200\ln\left|\ln\epsilon\right|}\right),
 \end{align*}one has
 \begin{align*}
 \sum_{|j|>j_0+N}|q_j(0)|^2\leq 2\delta.
 \end{align*}

\end{proof}

\section*{Acknowledgments}

   H.C.  was supported by NNSFC No. 11671066, No. 11401041 and NSFSP No. ZR2019MA062. Y.S.  was supported by NNSFC No.  11901010 and   Z.Z. was  supported by NNSFC
No. 11425103. The authors are very grateful to the anonymous referees for valuable suggestions.

\newcommand{\etalchar}[1]{$^{#1}$}


\begin{thebibliography}{DGPS99}

\bibitem[Aiz94]{Aiz94}
M.~Aizenman.
\newblock Localization at weak disorder: some elementary bounds.
\newblock {\em Rev. Math. Phys.}, 6(5A):1163--1182, 1994.
\newblock Special issue dedicated to Elliott H. Lieb.

\bibitem[AM93]{AM93}
M.~Aizenman and S.~Molchanov.
\newblock Localization at large disorder and at extreme energies: an elementary
  derivation.
\newblock {\em Comm. Math. Phys.}, 157(2):245--278, 1993.

\bibitem[And58]{And58}
P.~W. Anderson.
\newblock Absence of diffusion in certain random lattices.
\newblock {\em Physical review}, 109(5):1492, 1958.

\bibitem[AW15]{AW15}
M.~Aizenman and S.~Warzel.
\newblock {\em Random operators}, volume 168 of {\em Graduate Studies in
  Mathematics}.
\newblock American Mathematical Society, Providence, RI, 2015.
\newblock Disorder effects on quantum spectra and dynamics.

\bibitem[BFG88]{BFG88}
G.~Benettin, J.~Fr\"{o}hlich, and A.~Giorgilli.
\newblock A {N}ekhoroshev-type theorem for {H}amiltonian systems with
  infinitely many degrees of freedom.
\newblock {\em Comm. Math. Phys.}, 119(1):95--108, 1988.

\bibitem[BLS{\etalchar{+}}11]{66}
J.D. Bodyfelt, T.V. Laptyeva, Ch. Skokos, D.O. Krimer, and S.~Flach.
\newblock Nonlinear waves in disordered chains: {P}robing the limits of chaos
  and spreading.
\newblock {\em Phys. Rev. E}, 84(1):016205, 2011.

\bibitem[BW07]{BW07}
J.~Bourgain and W.-M. Wang.
\newblock Diffusion bound for a nonlinear {S}chr\"{o}dinger equation.
\newblock In {\em Mathematical aspects of nonlinear dispersive equations},
  volume 163 of {\em Ann. of Math. Stud.}, pages 21--42. Princeton Univ. Press,
  Princeton, NJ, 2007.

\bibitem[BW08]{BW08}
J.~Bourgain and W.-M. Wang.
\newblock Quasi-periodic solutions of nonlinear random {S}chr\"{o}dinger
  equations.
\newblock {\em J. Eur. Math. Soc. (JEMS)}, 10(1):1--45, 2008.

\bibitem[DGPS99]{DGPS}
F.~Dalfovo, S.~Giorgini, L.~P. Pitaevskii, and S.~Stringari.
\newblock Theory of {B}ose-{E}instein condensation in trapped gases.
\newblock {\em Reviews of Modern Physics}, 71(3):463, 1999.

\bibitem[DLS85]{DS85}
F.~Delyon, Y.~L\'{e}vy, and B.~Souillard.
\newblock Anderson localization for one- and quasi-one-dimensional systems.
\newblock {\em J. Stat. Phys.}, 41(3-4):375--388, 1985.

\bibitem[FKS08]{FKS08}
S.~Fishman, Y.~Krivolapov, and A.~Soffer.
\newblock On the problem of dynamical localization in the nonlinear
  {S}chr\"{o}dinger equation with a random potential.
\newblock {\em J. Stat. Phys.}, 131(5):843--865, 2008.

\bibitem[FKS09a]{FKS09}
S.~Fishman, Y.~Krivolapov, and A.~Soffer.
\newblock Perturbation theory for the nonlinear {S}chr\"{o}dinger equation with
  a random potential.
\newblock {\em Nonlinearity}, 22(12):2861--2887, 2009.

\bibitem[FKS09b]{12}
S.~Flach, D.O. Krimer, and Ch. Skokos.
\newblock Universal spreading of wave packets in disordered nonlinear systems.
\newblock {\em Phys. Rev. Lett.}, 102(2):024101, 2009.

\bibitem[FKS12]{FKS12}
S.~Fishman, Y.~Krivolapov, and A.~Soffer.
\newblock The nonlinear {S}chr\"{o}dinger equation with a random potential:
  results and puzzles.
\newblock {\em Nonlinearity}, 25(4):R53--R72, 2012.

\bibitem[FMSS85]{FMSS85}
J.~Fr\"{o}hlich, F.~Martinelli, E.~Scoppola, and T.~Spencer.
\newblock Constructive proof of localization in the {A}nderson tight binding
  model.
\newblock {\em Comm. Math. Phys.}, 101(1):21--46, 1985.

\bibitem[FS83]{FS83}
J.~Fr\"{o}hlich and T.~Spencer.
\newblock Absence of diffusion in the {A}nderson tight binding model for large
  disorder or low energy.
\newblock {\em Comm. Math. Phys.}, 88(2):151--184, 1983.

\bibitem[FSW86]{FSW86}
J.~Fr\"{o}hlich, T.~Spencer, and C.~E. Wayne.
\newblock Localization in disordered, nonlinear dynamical systems.
\newblock {\em J. Stat. Phys.}, 42(3-4):247--274, 1986.

\bibitem[GMP77]{GMP77}
I.~Goldseid, S.~Molchanov, and L.~Pastur.
\newblock A random homogeneous {S}chr\"{o}dinger operator has a pure point
  spectrum.
\newblock {\em Funct. Anal. Appl.}, 11(1):1--10, 96, 1977.

\bibitem[GYZ14]{GYZ14}
J.~Geng, J.~You, and Z.~Zhao.
\newblock Localization in one-dimensional quasi-periodic nonlinear systems.
\newblock {\em Geom. Funct. Anal.}, 24(1):116--158, 2014.

\bibitem[Kir08]{K08}
W.~Kirsch.
\newblock An invitation to random {S}chr\"{o}dinger operators.
\newblock In {\em Random {S}chr\"{o}dinger operators}, volume~25 of {\em Panor.
  Synth\`eses}, pages 1--119. Soc. Math. France, Paris, 2008.
\newblock With an appendix by Fr\'{e}d\'{e}ric Klopp.

\bibitem[PS08]{10}
A.S. Pikovsky and D.L. Shepelyansky.
\newblock Destruction of {A}nderson localization by a weak nonlinearity.
\newblock {\em Phys. Rev. Lett.}, 100(9):094101, 2008.

\bibitem[SF10]{67}
Ch. Skokos and S.~Flach.
\newblock Spreading of wave packets in disordered systems with tunable
  nonlinearity.
\newblock {\em Phys. Rev. E}, 82(1):016208, 2010.

\bibitem[SKKF09]{14}
Ch. Skokos, D.O. Krimer, S.~Komineas, and S.~Flach.
\newblock Delocalization of wave packets in disordered nonlinear chains.
\newblock {\em Phys. Rev. E}, 79(5):056211, 2009.

\bibitem[SW86]{SW86}
B.~Simon and T.~Wolff.
\newblock Singular continuous spectrum under rank one perturbations and
  localization for random {H}amiltonians.
\newblock {\em Comm. Pure Appl. Math.}, 39(1):75--90, 1986.

\bibitem[Wan08]{W08}
W.-M. Wang.
\newblock Logarithmic bounds on {S}obolev norms for time dependent linear
  {S}chr\"{o}dinger equations.
\newblock {\em Comm. Partial Differential Equations}, 33(10-12):2164--2179,
  2008.

\bibitem[WZ09]{WZ09}
W.-M. Wang and Z.~Zhang.
\newblock Long time {A}nderson localization for the nonlinear random
  {S}chr\"{o}dinger equation.
\newblock {\em J. Stat. Phys.}, 134(5-6):953--968, 2009.

\bibitem[Yua02]{Yua02}
X.~Yuan.
\newblock Construction of quasi-periodic breathers via {KAM} technique.
\newblock {\em Comm. Math. Phys.}, 226(1):61--100, 2002.

\end{thebibliography}
\end{document}